\documentclass{amsart}%
\usepackage{amsmath,amssymb}
\newcommand{\xa}{x_p}

\newcommand{\kjn}{U_{k+j:n}}
\newcommand{\dkjn}{\Delta_{k+j,n}}
\theoremstyle{plain}
\newtheorem{theorem}{Theorem}%[section]
\theoremstyle{definition}
\newtheorem{lemma}{Lemma}%[section]
\theoremstyle{remark}

\begin{document}
	\title{Spacings Around An Order Statistic}
	
	\author{H. N. Nagaraja}
	\address{Division of Biostatistics, The Ohio State University\\1841 Neil Avenue, Columbus, OH 43210, USA.}
	   \email{nagaraja.1@osu.edu}
	
\author{Karthik Bharath}
\address{School of Mathematical Sciences, University of Nottingham, \\ University Park, Nottingham NG72RD, UK}
\email{karthikbharath@gmail.com}
%\address{Department of Statistics, The Ohio State University, \\1958 Neil Ave,  Columbus, OH 43210, USA}

\author{Fangyuan Zhang}
\address{Department of Statistics, The Ohio State University, \\1958 Neil Ave,  Columbus, OH 43210, USA}
 \email{zhang.1243@buckeyemail.osu.edu}
	\maketitle
	
	\begin{abstract}
	We determine the joint limiting distribution of adjacent spacings around a central, intermediate, or an extreme order statistic $X_{k:n}$ of a random sample of size $n$ from a continuous distribution $F$. For central and intermediate cases, normalized spacings in the left and right neighborhoods are asymptotically i.i.d. exponential random variables.  The associated independent Poisson arrival processes are independent of $X_{k:n}$. For an extreme $X_{k:n}$, the asymptotic independence property of spacings fails for $F$ in the domain of attraction of Fr\'{e}chet and Weibull ($\alpha \neq 1$) distributions. This work also provides additional insight into the limiting distribution for the number of observations around $X_{k:n}$ for all three cases.\\
		\\
		%\begin{keywords}
	\end{abstract}

%%%%%
\section{Introduction}
Let $X_{1:n} \leq X_{2:n}\leq\cdots \leq X_{n:n}$ be order statistics of a random sample $X_1,\ldots,X_n$ from a continuous cdf $F$. For $1 \leq k \leq n$, we examine the clustering of data around the order statistic $X_{k:n}$.  This is done by an investigation into the limiting properties of the right and left neighborhoods formed by the adjacent spacings    $(X_{k+1:n}-X_{k:n}, \ldots, X_{k+r:n}-X_{k+r-1:n})$ and $(X_{k:n}-X_{k-1:n},\ldots, X_{k-s+1:n}-X_{k-s:n})$ for fixed $r$ and $s$.  We let $n \to \infty$ and consider three scenarios: (i) Central case where $k/n \to p$, $0 < p < 1$; (ii) Intermediate case where $k, n-k \to \infty$ and $k/n \to 0$ or $1$; (iii) Extreme case where $k$  or $n-k$ is held fixed.  In the first two cases we show that, under some mild assumptions,  these $(r+s)$ spacings appropriately scaled with a common scale parameter converge weakly to  a set of i.i.d. standard exponential random variables (rvs). In the extreme case,  this conclusion holds only when $F$ is in the domain of attraction of the Gumbel cdf $G_3$, or the Weibull type cdf $G_{2; \alpha}$ with $\alpha = 1$.  A direct and useful consequence of such a result is that  order statistics around a selected one arrive as in a homogeneous Poisson  process.

Neighborhoods around a selected order statistic have been investigated by several authors in recent years.  Almost all these results, starting with $X_{n:n}$,  have concentrated on the distribution of counts around it. We refer to a few, relevant to our results, from an exhaustive list: Balakrishnan and Stepanov (2005); Dembi\'nska et al.\ (2007); Pakes and Steutel (1997); Pakes (2009); and Dembi\'nska and Balakrishnan (2010).
These authors typically consider neighborhoods of the form $(X_{k:n}-d, X_{k:n})$ or $(X_{k:n}, X_{k:n}+d)$ where the lengths of the intervals may or may not depend on $n$; in some papers the $d$'s are induced by the quantile function $F^{-1}$ or are chosen to be random.  While these approaches are beneficial from a technical perspective, it is more natural and practical to consider neighborhoods that are in the scale of the data collected.  This is our motivation for considering the joint distribution of adjacent spacings. Our approach allows us to characterize the process governing the distribution of counts and provides additional insight into the asymptotic properties of the counts of cluster sizes around a specified order statistic. 
%%%%%%%%%

Section 2 contains preliminaries that explore the properties of uniform and exponential order statistics; it introduces the von Mises conditions and the associated extreme value distributions.  Section 3 is concerned with the joint distribution of a central order statistic and spacings adjacent to it on its right and left neighborhoods.  The Poisson arrival process of adjacent order statistics is established there.  Assuming von Mises conditions, Section 4 reaches a similar conclusion for the neighborhood of an intermediate order statistic. Section 5 displays the distributional structure of the extreme spacings assuming $F$ is in the domain of attraction of an extreme value distribution.  Section 6 applies our results and describes the limiting distribution of the counts of observations around an order statistic.  Section 7 discusses further applications of our results and contains concluding remarks.
%%%%%%%%%%

Let $f(x)$ denote the pdf and $F^{-1}(p)$, $0 \leq p \leq 1$, be the quantile function associated with $F(x)$, where $F^{-1}(p) = \inf\{x: F(x) \geq p\}$ for $0<p \leq 1$, $F^{-1}(0) = \sup\{x: F(x) =0\}$.  We interchangeably use $x_p$ and $F^{-1}(p)$ as the $p^{th}$ quantile. It is well-known that if $F$ is differentiable at $x_p$ with finite and positive pdf $f(x_p)$, $F^{-1}$ is differentiable at $p$ with derivative $1/f(x_p)$.  Standard uniform and exponential rvs are respectively denoted by $U$ and $Z$.   An exponential rv with rate parameter $\lambda$ will be denoted by Exp($\lambda$), and  Poi($\lambda$) represents a Poisson rv with mean parameter $\lambda$.  The sum of $r$ i.i.d. standard exponentials is a Gamma rv, to be denoted as Gam($r$). A Weibull rv with shape parameter $\delta$ will be denoted by Wei($\delta$).  Further, a standard normal rv will be denoted by $N(0,1)$ and its pdf by $\phi(\cdot)$. The $Z_i$'s and $Z^*_i$'s are i.i.d. Exp($1$) rvs. The symbol $\sim$ indicates asymptotic equivalence.

The rv $U_{i:n}, 1 \leq i \leq n$, is the $i^{th}$ order statistic from a random sample of size $n$ from a standard uniform population.  The distributional equivalence, $X_{i:n} \stackrel{d}{=}F^{-1}( U_{i:n})$, for {\it any} collection of order statistics from an arbitrary cdf $F$ is helpful in our investigations.

\section{Preliminaries}
\subsection{Spacings near a Uniform Order Statistic}
The key to our approach is the following well-known exchangeable property of the uniform order statistics.   Let $U_{0:n} = 0$ and $U_{n+1:n}= 1$, and define the {\it uniform spacing}
\begin{equation}\label{unifsp0}
\Delta_{i,n} = U_{i+1:n}-U_{i:n}, 0 \leq i \leq n. 
\end{equation} Then it is well known that the $\Delta_{i,n}$'s are exchangeable and for any fixed $r$, and for constants $v_i \geq 0, i = 1, \ldots r$ with $r \leq n$, and $\sum_{i=1}^r v_i \leq 1$, the joint survival function of $\Delta_{1,n}, \ldots, \Delta_{r,n}$ (and hence any collection of $r$ $\Delta_{i,n}$'s) is given by (see, e.g., David and Nagaraja, 2003, p.\ 135)
\begin{equation*}
P(\Delta_{1,n} > v_1, \ldots, \Delta_{r,n} > v_r) = (1-\sum_{i=1}^r v_i)^{n-1}.
\end{equation*}
This means
\begin{equation}\label{unifsp1}
P(n\Delta_{1,n} > v_1, \ldots, n\Delta_{r,n} > v_r) \to \prod_{i=1}^r \left\{e^{-v_i}\right\}, v_1, \ldots, v_r > 0.
\end{equation}
That is, $n\Delta_{i,n}$ form an i.i.d. Exp($1$) sequence $Z_i$ as $n \to \infty$. The convergence is fast;  Problem P.5.19 of Reiss (1989, p.\ 201) notes that  there exist a constant $C$ such that for every positive integer $n$ and $r \leq n$,
$$\sup_{\mathcal{B} \in \mathbb{B}}|P\{(n\Delta_{1,n}, \ldots, n\Delta_{r,n}) \in \mathcal{B}\} - P\{(Z_1, \ldots, Z_r) \in \mathcal{B}\}| \leq C \cdot (r/n),$$
where $\mathbb{B}$ denotes the family of all Borel sets.  We record the implications of (\ref{unifsp1}) and the exchangeability of the $\Delta_{i,n}$'s as a lemma given below; it uses the fact that the inter-arrival times being i.i.d.\ Exp($\lambda$) rvs is a defining property of a homogeneous Poisson process with rate $\lambda$.\begin{lemma}\label{lemma1} Let $U_{i:n}$ denote the $i^{th}$ order statistic from a random sample of size $n$ from a standard uniform distribution, and  assume $n \to \infty$.  Then for any $k$  such that $n-k \to \infty$, \begin{equation*}
(n(U_{k+1:n}-U_{k:n}), \ldots n(U_{k+r:n}-U_{k+r-1:n})) \stackrel{d}{\to} (Z_1, \ldots, Z_r),
\end{equation*}
for any fixed $r$, and for any $k \to \infty$
\begin{equation*}
(n(U_{k:n}-U_{k-1:n}), \ldots n(U_{k-s+1:n}-U_{k-s:n})) \stackrel{d}{\to} (Z^*_1, \ldots, Z^*_s),
\end{equation*}
for any fixed $s$, where the $Z_i$'s and $Z^*_i$'s are all mutually independent Exp($1$) rvs.  That is, inter-arrival times of successive order statistics in the right and left neighborhoods of $k^{th}$ uniform order statistic, upon scaling by $n$, produce asymptotically independent homogeneous Poisson processes if $n$, $k$, and $n-k$  approach infinity.  If $k$ [resp.\ $n-k$] is bounded, the right [resp.\ left] neighborhood produces a Poisson process in the limit.
\end{lemma}
\subsection{Spacings near an Exponential Order Statistic}
When $F$ is standard exponential it is well-known that
\begin{equation}\label{exp_os}
X_{i:n} \stackrel{d}{=} \frac{Z_1}{n} + \cdots + \frac{Z_{i}}{n-i+1}, i = 1, \ldots, n,
\end{equation}
where the $Z_i$'s are i.i.d. Exp(1) rvs.  From this representation it follows that $(n-i+1)(X_{i:n} - X_{i-1:n})$ are i.i.d. Exp(1) rvs.  So, if $k/n \to p, 0 \leq p < 1$, $n(1-p)(X_{k+j:n} - X_{k+j-1:n}),$ $j= 1, \ldots, r$ will be asymptotically i.i.d Exp(1) rvs. 

When $(n-k)$ is bounded, $(n-k-j+1)(X_{k+j:n} - X_{k+j-1:n}),$ $j= 1, \ldots, r$ for $r \leq n-k$ turn out to be i.i.d.\ Exp(1) rvs.   Only in this scenario, we need finite and distinct scaling constants for the spacings to transform them into i.i.d.\ exponential rvs for any $n$, and hence asymptotically as well.  
\subsection{Extremes and Von Mises Conditions}
Suppose there exist sequences of constants $a_n$ and $b_n > 0$ such that $P\{(X_{n:n} - a_n)/b_n \leq x\}$ converges to  a nondegenerate cdf $G(x)$ corresponding to a rv $W$. Then we say $F$ is in the domain of maximal attraction to $G$ and we write $F \in \mathcal{D}(G)$.  Then it is known that $G$ is necessarily of one of the three types given below.  
%%%%%%%
\begin{eqnarray}\label{extremeG}
\mbox{(Fr\'{e}chet)}~~G_{1;\alpha}(x) &=&\begin{cases}0 &x \leq 0, ~~\alpha >0, \\
\exp(-x^{-\alpha})& x > 0;
\end{cases}\nonumber \\[.1in]
\mbox{(Weibull)}~~G_{2;\alpha}(x) &=&
\begin{cases}\exp\left[-(-x)^{\alpha}\right]
&x \leq 0, \alpha > 0,\\
1&x > 0; \end{cases} \\[.1in]
\mbox{(Gumbel)}~~~~~G_3(x)&=&\exp\left({-e^{-x}}\right), -\infty < x < \infty. \nonumber
\end{eqnarray}
%%%%%%%%%%%%
The following are necessary and sufficient conditions on the right tail of $F$ in order that $F \in \mathcal{D}(G)$. The first two are due to Gnedenko (1943) and the last one is due to de Haan (1970).  
 \begin{itemize}
 \item[(a)] $F \in \mathcal{D}(G_{1;\alpha})$ iff for all $t > 0$,
\begin{equation*}\label{nsc1}
\lim_{x \to \infty} \frac{1-F(tx)} {1- F(x)}= t^{-\alpha}.
\end{equation*} 
\item[(b)] $F \in \mathcal{D}(G_{2;\alpha})$ iff $x_1 (= F^{-1}(1))$  is finite, and the following condition holds for every $t > 0$:
\begin{equation*}\label{nsc2}
\lim_{x \to 0+} \frac{1- F(x_1 -
tx)}{1-F(x_1 - x)} = t^{\alpha}.
\end{equation*}

\item[(c)] $F \in \mathcal{D}(G_3)$  iff the following hold: 
$E(X|X > c)$ is finite for some $c$, and for all real $t$,
\begin{equation*}\label{nsc3} \lim_{x \to x_1} \frac{1- F(x + t m(x))}{1-F(x)} = e ^{-t},
\end{equation*}
where $m(x) = E(X-x | X > x)$.

\end{itemize}
Our approach for the intermediate case assumes the following sufficient conditions that are applicable to absolutely continuous cdf's.  The first two are due to  von Mises (1936), and the last one is due to Falk (1989) and is weaker than the corresponding von Mises condition that assumes differentiability of the pdf $f$ (see, e. g., David and Nagaraja, 2003, p.\ 300).
 \begin{enumerate}
\item[(a)]~ $F \in \mathcal{D}(G_{1;\alpha})$ if $f(x) > 0$ for all large $x$ and for some
$\alpha
> 0$,
\begin{equation}\label{sG1} \lim_{x \to \infty}\frac{ x f(x)}{ 1-F(x)} = \alpha.
\end{equation}
 \item[(b)]~ $F \in \mathcal{D}(G_{2;\alpha})$ if $x_1 < \infty$ and for some $\alpha
> 0$,
\begin{equation}\label{sG2} \lim_{x \to x_1-}\frac{(x_1 - x) f(x)}{1-F(x)} =
\alpha.
\end{equation}
 \item[(c)]~ $F \in \mathcal{D}(G_3)$ if $f(x) > 0$ for all $x$ in $(c, x_1)$ and 
$E(X|X > c)$ is finite for some $c$, and \begin{equation}\label{sG3} \lim_{x \to x_1-} \frac{f(x) m(x)}{1-F(x)} = 1,
\end{equation}
where $m(x) = E(X-x | X > x)$. 
\end{enumerate}
The family of limiting distribution for  normalized $X_{1:n}$ correspond to that of $-W$ where $W$ has one of the above three types of cdfs; parallel necessary and sufficient, and sufficient conditions exist that impose conditions on the left tail of $F$.

\section{Spacings Around a Central Order Statistic}
\subsection{Joint Distribution of Spacings}

For $0<p<1$, $X_{k:n}$ is a central order statistic if $\frac{k}{n} \to p$. For such an $X_{k:n}$, Smirnov (1952; Theorem 3, p.\ 12) has shown (as pointed out by a reviewer) that
\begin{equation}\label{smirnov}X_{k:n}\overset{a.s.} \to x_p
\end{equation}
if the condition $F(x) = p$ has a unique solution $x_p$. Since for any fixed $j$, \begin{equation}\label{link}
n(X_{k+j+1:n} - X_{k+j:n})\stackrel{d}{=} \frac{F^{-1}(U_{k+j+1:n}) - F^{-1}(U_{k+j:n}) }{(U_{k+j+1:n}-U_{k+j:n})} n(U_{k+j+1:n}-U_{k+j:n}),
\end{equation}
the limiting joint distribution of the spacings from an arbitrary cdf $F$ can be linked to that of a collection of i.i.d.\ standard uniform rvs provided the first factor on the right in (\ref{link}) above converges in probability to a nonzero constant.  

From (\ref{smirnov}) it follows that $\Delta_{k+j, n}=U_{k+j+1:n}-U_{k+j:n}$ (defined in (\ref{unifsp0})) almost surely converges to 0. The first factor on the right in (\ref{link}), \begin{equation}\label{scale1}
%\frac{F^{-1}(U_{k+j+1:n}) - F^{-1}(U_{k+j:n}) }{(U_{k+j+1:n}-U_{k+j:n})} = 
\frac{F^{-1}(U_{k+j:n} +\Delta_{k+j, n}) - F^{-1}(U_{k+j:n}) }{\Delta_{k+j, n}} \stackrel{a.s.}{\to} \frac{1}{f(x_p)},
\end{equation}
if the following condition holds:\\
\begin{equation}\label{condf}
f~\mbox{is positive, finite and continuous at}~~ x_p.
\end{equation}
This conclusion follows from the definition of the derivative of $F^{-1}$ and its assumed continuity at $p$.

 Upon using (\ref{scale1}), (\ref{link}), Slutsky's Theorem, and  Lemma \ref{lemma1} we conclude that jointly
\[
nf(\xa)(X_{k+j+1:n}-X_{k+j:n}) \overset{d}\to Z_{j+1}, j = -s, \ldots 0, \ldots, r-1,
\]
where the $Z_j$'s are i.i.d.\ Exp($1$) rvs if (\ref{condf}) holds. 

We can weaken the continuity assumption for $f$ in (\ref{condf}) with the following condition:
\begin{equation}\label{demb}
\lim_{(p_1,h)\to (p, 0^+)} \frac{F^{-1}(p_1+h)-F^{-1}(p_1)}{F^{-1}(p+h)-F^{-1}(p)}=1,
\end{equation}
where $0<p<1$. This assumption is similar to (17) in Dembi\'nska et al. (2007) (given as (\ref{Demb2}) in Section 6 later).  The condition (\ref{condf}) implies that (\ref{demb}) holds since the latter is satisfied upon dividing the numerator and denominator by $h$ and taking the double limit; the converse is not true. 

On the other hand, we can weaken the requirement for a finite nonzero $f(\xa)$ by modifying a condition on $F$  used by  Chanda (1975) [see also M. Ghosh and Sukhatme (1981)].  We assume that
\begin{equation}\label{den}
\lim_{h \to 0} \frac{|F^{-1}(p+ h)-F^{-1}(p)|}{|h|^{\theta}}=M(p,\theta) \in (0,\infty),
\end{equation}
for some $\theta>0$. 
If $f$ is indeed finite and nonzero at $x_p$, then the above condition is satisfied with $\theta=1$. Whenever $f(x_p)$ is finite and positive or (\ref{den}) holds, there is a unique solution to $F(x) = p$ and (\ref{smirnov}) holds.

Based on the above discussion, we can now formally state the result for the central case. 
\begin{theorem}\label{theorem1}
Let  $k/n \to p \in (0,1)$, and  $r$ and $s$ be fixed positive integers. 
\begin{enumerate}
\item[(a)] If condition (\ref{condf}) holds, or  if  (\ref{demb}) holds and $f(x_p)$ is finite and positive, 
$$\left\{nf(\xa)(X_{k+j:n}-X_{k +j-1:n}), j = -(s-1), \ldots, 0, \ldots, r \right\}$$ $$ \overset{d}\to \{Z^*_s, \ldots, Z^*_1, Z_1, \ldots, Z_r\}$$
where the $Z$'s are i.i.d. Exp($1$) rvs. Thus the two counting processes defined by setting the $j^{th}$ event to occur respectively at times $$nf(\xa)\sum_{i=k}^{k+j-1} (X_{i+1:n}-X_{i:n}) ~~\mbox{and}~~nf(\xa)\sum_{i=k-1}^{k-j} (X_{i+1:n}-X_{i:n}) $$ converge weakly to two independent homogeneous Poisson processes with unit intensity. 
\item [(b)] Assume (\ref{demb}) and (\ref{den}) hold.   Then,
\begin{eqnarray}\label{th2}
\left\{n^{\theta}(X_{k+j:n}-X_{k +j-1:n})/M(p, \theta) , j = -(s-1), \ldots, 0, \ldots, r \right\}\nonumber\\
\overset{d}\to \left\{(Z^*_s)^{\theta}, \ldots, (Z^*_1)^{\theta}, (Z_1)^{\theta}, \ldots, (Z_r)^{\theta}\right\}.\end{eqnarray}
That is, the counting processes defined by setting the $j^{th}$ event of the process to occur at times $$n^{\theta}\sum_{i=k}^{k+j-1}  (X_{i+1:n}-X_{i:n})/M(p, \theta)~~\mbox{and}~~n^{\theta}\sum_{i=k-1}^{k-j} (X_{i+1:n}-X_{i:n})/M(p, \theta) $$ converge to i.i.d. renewal processes with Wei($1/\theta$) renewal distribution. They reduce  to  homogeneous Poisson processes with unit intensity only when $\theta=1$ and $f(x_p)$ is finite and positive.
\end{enumerate}
\end{theorem}
\begin{proof}
To prove part $(a)$, we need to show that (\ref{scale1}) holds whenever (\ref{condf}) holds, or  if  (\ref{demb}) holds and $f(x_p)$ is finite and positive.  Then we would use (\ref{scale1}), (\ref{link}), Slutsky's Theorem, and  Lemma \ref{lemma1}.  We have shown earlier that (\ref{scale1}) holds whenever (\ref{condf}) is satisfied. 
 
 If (\ref{demb}) holds and $f(x_p)$ is finite and positive, the left side expression in (\ref{scale1}) can be written as
\begin{equation*}\label{scale2}
\frac{F^{-1}(U_{k+j:n} +\Delta_{k+j,n}) - F^{-1}(U_{k+j:n}) }{{F^{-1}(p +\Delta_{k+j,n}) - F^{-1}(p) }} \cdot \frac{F^{-1}(p +\Delta_{k+j,n}) - F^{-1}(p) }{\Delta_{k+j,n}},\end{equation*}
where the first factor converges to 1 and the second factor converges to $1/f(x_p)$, both almost surely. Thus, (\ref{scale1})  is established.%%%%%%%%%%

For $(b)$, the idea is similar.  We note that 
$ n^{\theta}(F^{-1}(\kjn+\dkjn)-F^{-1}(\kjn))$ can be written as $$\frac{F^{-1}(\kjn+\dkjn)-F^{-1}(\kjn)}{\dkjn^\theta}(n \dkjn)^\theta.$$
Assumption (\ref{den}) coupled with (\ref{demb}) ensures that the first factor above converges almost surely  to  $M(p, \theta).$
Since $P(n \dkjn > w) \to \exp\{-w\}$ for all $w > 0$, 
$P((n \dkjn)^{\theta} > w) \to \exp\{-w^{1/\theta}\}$; that is,  
$ \left(n\dkjn\right)^\theta \overset{d} \to Z^\theta.$ 
The claim in (\ref{th2}) now readily follows.  If $\theta = 1$ and (\ref{den}) holds, then $M(p;1) = (F^{-1}(x_p))^{'} = 1/f(x_p)$ has to be positive and finite, and the limiting arrival process would be Poisson. \hfill$\Box$
\end{proof}
\noindent{\bf Remark.} The condition (\ref{den}) does not imply (\ref{demb}); nor does it ensure that $f(x_p)$ is finite and positive.  Consider the pdf  \[
f(x)={(\eta +1)}|x|^{\eta}/2, \qquad |x|\leq 1, \eta>-1.
\]
This is a corrected version of the pdf given in Chanda (1975), and discussed in Ghosh and Sukhatme (1981) (we thank a reviewer for noticing the error).  
The associated quantile function is given by
 \[
F^{-1}(u)=\left\{ 
  \begin{array}{l l}
    -(1-2u)^{1/(\eta+1)}, & \quad 0 < u \leq \frac{1}{2} \\[.1in]
~~(2u-1)^{1/(\eta+1)},& \quad \frac{1}{2} \leq u < 1.
  \end{array} \right.
\]
This quantile function fails to satisfy the condition in (\ref{demb}) when $p=0.5$ and $\eta$ is a positive even integer, but satisfies (\ref{den}) with $\theta={1}/{(\eta+1)}$. Here $\xa=0$, $M(p, \theta) =  2^{\theta}$; $f(\xa)$ is 0 or $\infty$ depending on whether $\eta>0$ or $\eta \in (-1,0)$.

\subsection{Asymptotic Independence of a Central Order Statistic and Spacings in Its Neighborhood}
We now assume $k/n= p+o(n^{-1/2})$ and establish the independence of $X_{k:n}$ and spacings around it. \subsubsection{The Uniform Parent}
Using the (well-known) joint pdf of the consecutive standard uniform order statistics $U_{k-s:n}, \ldots, U_{k:n}, \ldots, U_{k+r:n}$, we first obtain the joint pdf of  appropriately normalized $U_{k:n}$ and the vector $$\left\{n(U_{k+j:n}-U_{k +j-1:n}), j = -(s-1), \ldots, 0, \ldots, r \right\},$$
and thus determine the limiting form of the joint pdf. 

 The joint pdf of
 $U_{k-s:n},\cdots , U_{k+r:n}$ is given by
\[\frac{n!}{(k-s-1)!(n-k-r)!}u_{k-s}^{k-s-1}(1-u_{k+r})^{n-k-r},\]for $ 0 <u_{k-s}< \cdots < u_k < \cdots < u_{k+r}< 1.$
With $t_n = \sqrt{[p(1-p)]/{n}}$ consider the transformation
\begin{eqnarray}\label{vis}
v_0&=&(u_{k:n}-p)/t_n,\nonumber\\ v^*_1&=&n(u_{k:n}-u_{k-1:n}),\ldots, v^*_{s}=n(u_{k-s+1:n}-u_{k-s:n});\nonumber \\
v_1 &=&n(u_{k+1:n}-u_{k:n}),\ldots, v_{r} = n(u_{k+r:n}-u_{k+r-1:n}).\end{eqnarray}
Hence
\begin{eqnarray*}
  u_{k:n} &=& p+t_n v_0 \\
  u_{k+1:n}&=& p+t_n v_0+(1/n)v_1\\
  \cdots\\
  u_{k+r:n}&=&p+t_nv_0+(1/n)(v_1+v_2+\cdots+v_{r})\\
  \cdots\\
  u_{k-1:n}&=& p+t_n v_0-(1/n)v^*_1\\
  \cdots\\
  u_{k-s:n}&=&p+t_nv_0-(1/n)(v^*_1+v^*_2+\cdots+v^*_{s}),
\end{eqnarray*}
and the Jacobian is $\left|\frac{\partial \mathbf{u}}{\partial \mathbf{v}}\right|
={t_n/n^{r+s}}
$.
The joint pdf of $V_0, V_1, \ldots,V_{r},  V^*_1, \ldots,V^*_{s}$ is
\begin{eqnarray*}
  &&\frac{n!}{(k-s-1)!(n-k-r)!}\frac{t_n}{n^{r+s}}\{p+t_nv_0-(1/n)(v^*_1+v^*_2+\cdots+v^*_{s})\}^{k-s-1}\\
  &&\times\{1-p-t_nv_0-(1/n)(v_1+v_2+\cdots+v_{r})\}^{n-k-r},\\
&\sim& \frac{n!}{(k-1)!(n-k)!}{t_n p^s(1-p)^r} (p + t_n v_0)^{k-s-1}(1-p-t_nv_0)^{n-k-r}\\ 
&&\times \{1-\frac{(v^*_1+v^*_2+\cdots+v^*_{s})}{n(p + t_n v_0)}\}^{k-s-1}\\
  &&\times \{1-\frac{(v_1+v_2+\cdots+v_{r})}{n(1-p-t_nv_0)}\}^{n-k-r},\\
&=& \eta_1 \times \eta_2 \times \eta_3,\end{eqnarray*}
say.  Since $k = np + o(\sqrt{n})$ and $ t_n \to 0$,$$\eta_2 \to \exp\{-(v^*_1+v^*_2+\cdots+v^*_{s})\}~~\mbox{and}~~ \eta_3\to \exp\{-((v_1+v_2+\cdots+v_{r})\}$$   for fixed $r$ and $s$. Further,
\begin{eqnarray*}
\eta_1 &=& \frac{n!}{(k-1)!(n-k)!}{t_n p^s(1-p)^r} (p + t_n v_0)^{k-s-1}(1-p-t_nv_0)^{n-k-r}\\
&\sim&\frac{n!}{(k-1)!(n-k)!}{t_n} (p + t_n v_0)^{k-1}(1-p-t_nv_0)^{n-k},
\end{eqnarray*}
the pdf of $V_0 = (U_{k:n} - p)/t_n$.  With $k = np + o(\sqrt{n})$, using Stirling's approximations for the factorials and the expansion $\log(1+x) = x - x^2/2 + o(x^2)$ for $x$ close to 0, it can be shown that this pdf  converges to $\phi(v_0)$.

The conclusion of the above discussion is summarized below.

\begin{lemma}\label{lem2} With $k = np + o(\sqrt{n})$, asymptotically (i)  $V_0, V_1, \ldots, V_r, V^*_1, \ldots, V^*_s$ are  mutually independent, (ii) $V_0$ is $N(0,1)$, and (iii) the remaining rvs are i.i.d. standard exponential, where the rvs involved are defined in (\ref{vis}).
\end{lemma}
\subsubsection{Arbitrary Parent}
 By establishing density convergence under the assumption that $k/n= p+o(n^{-1/2})$ we have shown above that 
\begin{equation}\label{norm1} \frac{\sqrt{n}}{\sqrt{p(1-p)}}(U_{k:n}-p) \overset{d} \rightarrow N(0,1).\end{equation}
The conclusion in (\ref{norm1}) also follows from J. K. Ghosh (1971) who has shown that if $f(x_p)$ is positive and finite, 
\begin{equation*}
\frac{\sqrt{n}f(x_p)}{\sqrt{p(1-p)}}(X_{k:n}-x_p) \overset{d} \rightarrow N(0,1).
\end{equation*}
We have shown in Section 3.1 that when $k/n= p+o(1)$, if condition (\ref{condf}) holds or  if  (\ref{demb}) holds and $f(x_p)$ is finite and positive, 
$$\frac{X_{k+j+1:n}-X_{k +j:n}}{U_{k+j+1:n}-U_{k +j:n}}=\frac{F^{-1}(U_{k+j+1:n})-F^{-1}(U_{k +j:n})}{U_{k+j+1:n}-U_{k +j:n}} \stackrel{a.s.}{\to} \frac{1}{f(x_p)}; $$  and if (\ref{demb}) and (\ref{den}) hold,
%%%%
$$\frac{X_{k+j+1:n}-X_{k +j:n}}{\Delta_{k+j:n}^{\theta}}=\frac{F^{-1}(U_{k+j:n} +\Delta_{k+j:n})-F^{-1}(U_{k +j:n})}{\Delta_{k+j:n}^{\theta}} \stackrel{a.s.}{\to} M(p, \theta), $$ 
%%%%%
for $ j = -s, \ldots, 0, \ldots, r -1.$
Further, $$\frac{X_{k:n}-x_p}{U_{k:n}-p}=\frac{F^{-1}(U_{k:n})-F^{-1}(p)}{U_{k:n}-p} \stackrel{a.s.}{\to} \frac{1}{f(x_p)}$$
whenever $f(x_p)$ is finite and positive; 
 $$\frac{|X_{k:n}-x_p|}{|U_{k:n}-p|^{\theta}}=\frac{|F^{-1}(U_{k:n})-F^{-1}(p)|}{|U_{k:n}-p|^{\theta}} \stackrel{a.s.}{\to} M(p, \theta),$$
  if (\ref{den}) holds. 

In view of Lemma {\ref{lem2}}, assuming $k = np + o(\sqrt{n})$ we have established the asymptotic independence of the normalized spacings $({X_{k+j:n}-X_{k +j-1:n}})$ introduced in Theorem \ref{theorem1} and appropriately normalized $X_{k:n}$ under the conditions stated there.   This discussion leads to the following result.

\begin{theorem}\label{theorem2}
Let $k/n= p+o(n^{-1/2})$, $0 <p < 1$, and $r$ and $s$ be fixed positive integers.
\begin{itemize}
\item[(a)]
 If condition (\ref{condf}) holds, or  if  (\ref{demb}) holds and $f(x_p)$ is finite and positive, 
$$\left(\frac{\sqrt{n}f(x_p)(X_{k:n}-x_p)}{\sqrt{p(1-p)}}, nf(\xa)(X_{k+j:n}-X_{k +j-1:n}),  -(s-1) \leq j \leq r \right)$$ $$ \overset{d}\to (N(0,1), Z^*_s, \ldots, Z^*_1, Z_1, \ldots, Z_r).$$
\item[(b)]
Assume (\ref{demb}) and (\ref{den}) hold.  Then
\begin{eqnarray*}
\hspace{-.2in}\left({\frac{\sqrt{n}}{\sqrt{p(1-p)}}}\left|\frac{(X_{k:n}-x_p)}{M(p,\theta)}\right|^{1/\theta},  n\left(\frac{X_{k+j+1:n}-X_{k +j:n}}{M(p, \theta)}\right)^{1/\theta} ,-s \leq j \leq r-1  \right)\\ \overset{d}\to (|N(0,1)|, Z^*_s, \ldots, Z^*_1, Z_1, \ldots, Z_r).\end{eqnarray*}
\end{itemize}
In both cases the $Z_i$'s and $Z^*_i$'s are Exp($1$) rvs, and the $r+s+1$ components in the limit vector are mutually independent.
\end{theorem}%%%%%%%%%%%%%%%%%%%%%%%%%
\subsection{Remarks - The Central Case}

Siddiqui (1960) considered higher order spacings around a central order statistic and showed that when $F$ is continuously twice differentiable and $f(x_p)$ is finite and positive, the rvs
\begin{equation*}\label{sidd}
\frac{\sqrt{n}f(x_p)(X_{k:n}-\xa)}{\sqrt{p(1-p)}}, \quad {nf(x_p)}(X_{k:n}-X_{k-s:n}) \quad \text{and  }   {nf(x_p)}(X_{k+r:n}-X_{k:n})
\end{equation*}
are asymptotically independent when $k/n \to p \in (0,1)$ with $r/n$ and $s/n$ tending to zero; further,  asymptotically the higher order spacings are Gam$(r)$ and Gam$(s)$, respectively.  We have proved a more refined result here with less assumptions on the properties of $F$,  but have taken $r$ and $s$ to be fixed.

 Pyke's (1965) classic paper on spacings shows (Theorem 5.1) that\\ $nf(x_{p_1})(X_{i:n}-X_{i-1:n})$ and $nf(x_{p_2})(X_{j:n}-X_{j-1:n})$ with $i/n \to p_1$ and $j/n \to p_2$ where $0<p_1 \neq p_2<1$ are asymptotically i.i.d. Exp($1$) rvs. The key difference is that the spacings considered there are far apart, while our focus is on adjacent spacings around $X_{i:n}$.
 
 The asymptotic half-normal distribution of the normalized central order statistic under the conditions of part (b) of Theorem \ref{theorem2} is comparable to Chanda's (1975) conclusion; our condition (\ref{den}) is on $F^{-1}$ whereas his comparable condition (given as  (6) there) is on $F$.%%%%%
\section{The Intermediate Case}
Here we lean on the work of Falk (1989) and directly examine the convergence of the joint pdf of an intermediate order statistic $X_{k:n}$ and spacings around it.  We assume $k \to \infty$ but $k/n \to 1$ as $n \to \infty$ such that $n-k \to \infty$ and assume one of the von Mises sufficient conditions stated in (\ref{sG1})--(\ref{sG3}) holds. Theorem 2.1 of Falk (1989) shows that 
\begin{equation}\label{intconsts}
(X_{k:n} - a_n)/b_n \stackrel{d}{\to} N(0,1)~~ \mbox{with}~~ a_n  = F^{-1}(k/n); b_n = \sqrt{n-k}/(n f(a_n)).
\end{equation}
This is established by showing that the pdf of $(X_{k:n}-a_n)/{b_n}$ at $x$,
\begin{equation}\label{os1}
\frac{n!}{(k-1)!(n-k)!}[F(a_n+b_nx)]^{k-1}[1-F(a_n+b_n x)]^{n-k}b_n f(a_n +b_n x) \to \phi(x),
\end{equation}
for all real $x$.

Consider the joint pdf of $X_{k-s:n}, \ldots,X_{k:n},\ldots , X_{k+r:n}$:
\[\frac{n!}{(k-s-1)!(n-k-r)!}[\prod_{j=1}^{s}f(x_j^*)][F(x^*_s)]^{k-s-1}f(x_0)[\prod_{j=1}^{r}f(x_j)][1-F(x_{r})]^{n-k-r},\]
for $x_s^*<\cdots < x^*_1<x_0 <x_1 < \cdots < x_r$.
Define
\begin{equation}
c_n=\frac{1}{nf(a_n)}=\frac{b_n}{\sqrt{n-k}}= \frac{1}{nf(F^{-1}(k/n))}\end{equation}
and consider the transformation
\begin{eqnarray}\label{yis}
y_0&=&{(x_{k:n}-a_n)}/{b_n}, \nonumber \\
y_1&=&(x_{k+1:n}-x_{k:n})/c_n,\ldots, y_{r} = (x_{k+r:n}-x_{k+r-1:n})/{c_n},\nonumber \\
y^*_1&=&(x_{k:n}-x_{k-1:n})/c_n,\ldots, y^*_{s}=(x_{k-s+1:n}-x_{k-s:n})/{c_n},
\end{eqnarray}
so that \begin{eqnarray*}
  x_{k:n} &=& a_n+b_n y_0; \\
  x_{k+j:n}&=& a_n+b_n y_0+c_n(y_1+ \cdots + y_j),~ j = 1, \ldots, r;\\
   x_{k-j:n}&=& a_n+b_n y_0-c_n(y^*_1+ \cdots + y^*_j),~ j = 1, \ldots, s.
\end{eqnarray*}
The Jacobian  is $\left|\frac{\partial \mathbf{x}}{\partial \mathbf{y}}\right|
={b_nc_n^{r+s}}$
and the joint pdf of $Y_0, Y_1, \ldots,Y_{r}, Y^*_1, \ldots,Y^*_{s}$, defined in (\ref{yis}), is
\begin{eqnarray}
  &&\hspace{-.4in}\frac{n!}{(k-1)! (n-k)!}[F(a_n+b_n y_0)]^{k-1}[1-F(b_n + a_ny_0)]^{n-k}b_nf(a_n + b_ny_0) \label{t1}\\
  &\times&\left(\frac{F(a_n+b_ny_0-c_n(y^*_1+\cdots+y^*_{s}))}{F(a_n+b_ny_0)}\right)^{k-s-1}\label{t2}\\
    &\times& \prod_{j=1}^s \left\{c_n (k - j)\frac{f(a_n + b_n y_0-c_n(y^*_1+\cdots + y^*_j))}{F(a_n + b_ny_0)} \right\} \label{t3} \\
  &\times& \prod_{j=1}^r \left\{c_n (n-k - j + 1)\frac{f(a_n + b_n y_0+c_n(y_1+\cdots + y_j))}{1-F(a_n + b_ny_0)} \right\} \label{t4}\\
  &\times& \left(\frac{1-F(a_n+b_ny_0+c_n(y_1+\cdots+y_{r}))}{1-F(a_n+b_ny_0)}\right)^{n-k-r} \label{t5}\\
  &=& \tau_1 \times \tau_2 \times \tau_3 \times \tau_4 \times \tau_5, \label{t_all}\end{eqnarray}
where $\tau_1-\tau_5$ are respectively given by (\ref{t1})-(\ref{t5}), $y_0$ is real, and $y_i, y^*_i > 0$.  

From (\ref{os1}) it follows that $\tau_1 \to \phi(y_0)$.  We will  establish the limits for remaining $\tau_i$ using the following lemma.
%%%%new%%%%%
\begin{lemma}\label{lemma3}
Suppose one of the von Mises conditions stated in (\ref{sG1})--(\ref{sG3}) holds, and  $n \to \infty$, $k/n \to 1$ such that $n-k \to \infty$. Then, with $a_n$ and $b_n$ given by (\ref{intconsts}), the following statements hold. 
\begin{itemize}
\item [(a)] For any real $y_0$, $\frac{1-F(a_n+b_ny_0)}{1-F(a_n)} = \frac{n(1-F(a_n+b_ny_0))}{n-k}\rightarrow 1$.\\
\item [(b)] If $c_n = b_n/\sqrt{n-k}$, for any $y_0$, $y_1$ real,  $\frac{f(a_n + b_ny_0+c_ny_1)}{f(a_n)}\rightarrow 1.$
\end{itemize}
\end{lemma}
\begin{proof}
(a) From Theorem 2.1 of Falk (1989) it follows that  (\ref{intconsts}) holds under the conditions we have assumed, and the limit distribution is $N(0,1)$.
From Theorem 1 of Smirnov (1967) [Remark(\romannumeral 2) of Falk (1989)] it then follows that
$[n-k+1 +n(F(a_n+b_n y_0)-1)]/\sqrt{n-k+1}\rightarrow x$ for all real $y_0$. Thus, $$\sqrt{n-k+1}\cdot\left\{1-\frac{1-F(a_n+b_ny_0)}{1-F(a_n)}\right\}\rightarrow y_0$$ since $1-F(a_n)=(n-k)/n$. This implies that $(1-F(a_n+b_n y_0))/(1-F(a_n))\rightarrow 1$ for any  real $y_0$. 

(b)   In the proof of his Theorem 2.1, Falk establishes that whenever 
one of the sufficient conditions stated in (\ref{sG1})--(\ref{sG3}) holds, for any real $y$ for which $F(a_n + b_n  y) \to 1$ (or equvivalently $a_n + b_n  y \to x_1$) as $n \to \infty$,
\begin{equation}\label{Falk}
f(a_n + b_n \theta y)/f(a_n) \to 1
\end{equation}
uniformly for $\theta \in (0,1)$ where $a_n$ and $b_n$ are given in (\ref{intconsts}).  Part (a) that we just proved implies that for any real $y_0$, $F(a_n + b_n  y_0) \to 1$ as $n \to \infty$.  Thus from (\ref{Falk}) it follows that $f(a_n + b_n \theta y_0)/f(a_n) \to 1$
 for all real $y_0$.

For large $n-k$ and real $y_1$,
$$y_0+\frac{y_1}{\sqrt{n-k}} \in(0, 2y_0) ~\mbox{if} ~y_0>0; y_0+\frac{y_1}{\sqrt{n-k}}\in (2y_0, 0) ~\mbox{if}~ y_0 < 0.$$  Using (\ref{Falk}) with $y=2y_0$ we conclude that ${f(a_n +2y_0\theta b_n)}\{{f(a_n)}\}^{-1}\rightarrow 1$ uniformly  for all $0\le \theta<1$.  Since $a_n+b_ny_0+c_n y_1 = a_n+b_n(y_0+ y_1/{\sqrt{n-k}} )$ is in $(a_n, a_n +2y_0b_n)$ if $y_0 > 0$ and in $(a_n +2y_0b_n, a_n)$ if $y_0 < 0$.  Hence ${f(a_n + b_ny_0+c_ny_1)}/{f(a_n)}\rightarrow 1$ for all real $y_0 \neq 0$ and real $y_1$.
When $y_0=0$, for any real $y_1$, ${f(a_n + b_n({1}/{\sqrt{n-k}})y_1)}/{f(a_n)}\rightarrow 1$ since ${1}/{\sqrt{n-k}} \in (0, 1)$ and (\ref{Falk}) holds.  This completes the proof of the claim in (b).\hfill $\Box$
\end{proof}
%%%%%

With $y = y^*_1+\cdots+y^*_{s} >0$ consider the following component of $\tau_2$ in (\ref{t2}):
\begin{eqnarray*}
\frac{F(a_n+b_n y_0-c_n y)}{F(a_n+b_ny_0)}
  &=& 1-\frac{F(a_n+b_ny_0)-F(a_n+b_ny_0-c_ny)}{F(a_n+b_ny_0)}\\
&=&1-\frac{f(a_n+b_n y_0-\theta^*c_n y)}{F(a_n+b_n y_0)}\cdot c_n y\\
&=&1- \frac{yd_{k,n}}{k}
\end{eqnarray*}
where the second form above follows from the mean value theorem,  $\theta^*\in(0,1)$, and $$d_{k, n} =\frac{1}{F(a_n+b_n y_0)}\frac{f(a_n+b_n y_0+\theta^*c_n(-y))}{f(a_n)}\frac{k}{n}, $$ where we have used the fact that $c_n = 1/n f(a_n)$. From part (a) of Lemma \ref{lemma3} the first factor of $d_{k,n}$ above converges to 1; from part (b), the second factor approaches 1, and from our assumptions about $k$ and $n$ made in the intermediate case, the third factor also approaches 1.  Hence $d_{k, n} \to 1$ as $n, k \to \infty$.
Thus upon recalling (\ref{t2}) we obtain
$$
\tau_2 =\left(1- \frac{y d_{k,n}}{k}\right)^{k-s-1} \to \exp\{-y\} = \exp\{-(\sum_{j=1}^s y^*_j)\}, ~y^*_j > 0.$$

 With $y=y^*_1+\cdots +y^*_j$, the $j^{th}$ term in the product representing $\tau_3$ in (\ref{t3})  is given by
$$c_n (k - j )\frac{f(a_n + b_n y_0+c_n (-y))}{F(a_n + b_ny_0)} = \frac{k-j}{n} \frac{f(a_n + b_n y_0+c_n(-y))}{f (a_n)} \frac{1}{F(a_n + b_ny_0)}. $$
Using Lemma \ref{lemma3} as we did in proving $d_{k, n} \to 1$ as $n$ and $n-k \to \infty$, we conclude that the $j^{th}$ factor  of $\tau_3 \to 1$ for all $j$ and so does $\tau_3$.
%%%%%%%

%%%%%%%%%
With $y=y_1+y_2+\cdots +y_j$, the $j^{th}$ term in the product representing $\tau_4$ in (\ref{t4})   is given by
$$\hspace{-1.5in}\left\{c_n (n-k - j + 1)\frac{f(a_n + b_n y_0+c_n y)}{1-F(a_n + b_ny_0)} \right\}$$ $$\hspace{.5in} = \frac{n-k-j+1}{n(1-F(a_n))}\cdot \frac{f(a_n + b_n y_0+c_ny)}{f (a_n)} \cdot\frac{1- F(a_n)}{1-F(a_n + b_ny_0)}. $$
Since $F(a_n) = k/n$, the first factor above converges to 1, and Lemma \ref{lemma3} shows that the other two factors also approach 1 as $n$ and $n-k \to \infty$.  Thus $\tau_4 \to 1$.
%%%%%%%

Finally with $y = y_1+\cdots+y_{r}$ consider the following component of $\tau_5$ in (\ref{t5}):
$$
\frac{1-F(a_n+b_n y_0+c_n y)}{1-F(a_n+b_ny_0)}
  = 1-\frac{F(a_n+b_ny_0+c_ny)-F(a_n+b_ny_0)}{1-F(a_n+b_ny_0)}
$$
  \begin{eqnarray*}
&=&1-\frac{f(a_n+b_n y_0+\theta^*c_n y)}{1-F(a_n+b_n y_0)}\cdot c_n y~~\mbox{from the Mean Value Theorem}\\
&=& 1- y \frac{f(a_n+b_n y_0+\theta^*c_n y)}{f(a_n)}\cdot\frac{1-F(a_n)}{1-F(a_n+b_n y_0)}\cdot\frac{1}{n(1-F(a_n))} 
\end{eqnarray*}
where $\theta^*\in(0,1)$ and we have used the fact that $c_n = 1/n f(a_n)$. Lemma \ref{lemma3} implies that the first two factors of $y$ converge to 1 as $n, n-k \to \infty$, and the denominator of the last factor, $n(1-F(a_n)),$ is $n-k$.
Hence, $$\tau_5\sim(1-\frac{y_1+\cdots+y_{r}}{n-k})^{n-k-r}\rightarrow e^{-(y_1+\cdots+y_{r})},  y_1, \ldots, y_r > 0,$$ as $n, n-k \rightarrow\infty.$

Thus, we have formally proved the following theorem.
\begin{theorem}\label{theorem3}
Whenever one of the Von Mises condition (\ref{sG1})--(\ref{sG3}) holds, and $n, k,  n-k \to \infty$ such that $k/n \to 1$, with $a_n = F^{-1}(k/n)$ and $c_n = 1/nf(a_n)$,
$$((X_{k-s:n}-X_{k-s+1:n})/c_n,\cdots,(X_{k:n}-X_{k-1:n})/c_n, (X_{k+1:n}-X_{k:n})/c_n,\cdots,$$ $$\cdots, (X_{k+r:n}-X_{k+r-1:n})/c_n) \stackrel{d}{\rightarrow}(Z^*_s,\cdots, Z^*_{1}, Z_1,\cdots, Z_{r});$$
$$\frac{1}{c_n\sqrt{n-k}}(X_{k:n}-a_n) \stackrel{d}{\rightarrow} N(0,1),$$
where the $Z^*_i$'s and $Z_i$'s are Exp($1$) rvs, and the $r+s+1$ limiting rvs are mutually independent.
\end{theorem}
\subsection{Remarks - The Intermediate Case}

When $F \in \mathcal{D}(G_{1;\alpha})$,
$$\frac{a_n f(a_n)}{1-F(a_n)}=\frac{F^{-1}(k/n) c_n}{n-k}\rightarrow \alpha$$ and hence  ${\alpha (n-k)}/{F^{-1}(k/n)}$ can be chosen as $c_n$. When $F \in \mathcal{D}(G_{2;\alpha})$, the Von Mises condition implies $\alpha(n-k)/(x_1-F^{-1}({k}/{n}))$ can be used as $c_n$.
When $F \in \mathcal{D}(G_3)$, $[{nf(a_n)m(a_n)}/{(n-k)}]\rightarrow 1$ and we can use $m(a_n)/(n-k)$ as our $c_n$.

From Theorem \ref{theorem3} it follows that as in the central case, asymptotically, any two spacings, possibly of higher order, formed by nonoverlaping collections of order statistics around an intermediate order statistic $X_{k:n}$ are independent, and the collection is independent of $X_{k:n}$. 

Teugels (2001) has  introduced a family $\mathcal{C^*}$ of cdfs $F$ with the following property:
$F$ has an ultimately positive pdf $f$ and for all real $y$,
\begin{equation*}\label{teugels}
\frac{1}{h(x)}\left\{\frac{F\left(x+ y h(x)\frac{1-F(x)}{f(x)}\right)-F(x)}{1-F(x)}\right\} \to y,
\end{equation*}
whenever $h(x) \to 0$ as $x \to x_1$.  He states that the condition $F \in \mathcal{C^*}$  `slightly generalizes'  Falk's (1989) version of von Mises conditions (i.e., (\ref{sG1})--(\ref{sG3})).  Assuming $F \in \mathcal{C^*}$, Teugels shows that upon normalization described above (i) $X_{k:n}$ is asymptotically normal and (ii) $(X_{k:n} - X_{k-s:n})$ is asymptotically Gam$(s)$.  Their joint distribution and the asymptotic independence are not discussed there.  
\section{The Upper Extreme Case}
We now assume that $k \to \infty$ such that $n-k$ is fixed.
It is well-known that when $F \in D(G)$ for $G$ given in (\ref{extremeG}), 
\begin{equation}\label{multext} \left( ({X_{n:n} - a_n})/{b_n}, \cdots,  ({X_{k:n} - a_n})/{b_n}\right) \stackrel{d}{\to} (W_1, \cdots, W_{n-k+1})\end{equation}
where for any finite fixed $j$ the vector $(W_1, \cdots, W_j)$ has the same joint distribution as
%%%%%%%%%%%
 \begin{eqnarray}
& \left(Z_1^{-1/\alpha}, (Z_1+Z_2)^{-1/\alpha},\ldots, (Z_1+\cdots+Z_j)^{-1/\alpha}\right), ~\mbox{if}~ G = G_{1;\alpha};\label{joint_ext1}\\
& \left(-Z_1^{1/\alpha}, -(Z_1+Z_2)^{1/\alpha},\ldots, -(Z_1+\cdots+Z_j)^{1/\alpha}\right),  ~\mbox{if}~ G = G_{2;\alpha};\label{joint_ext2}
\end{eqnarray}
and if $G = G_{3}$,
\begin{equation}\label{joint_ext3a}
(W_1, \cdots, W_j)\stackrel{d}{=} \left(-\ln Z_1,-\ln(Z_1+Z_2),\ldots,-\ln(Z_1+\cdots+Z_j)\right)
\end{equation}
\begin{equation}\label{joint_ext3b}
\stackrel{d}{=}\left(\sum_{i=k}^{\infty} \frac{Z_i -1}{i} + \gamma - \sum_{i=1}^{k-1} \frac{1}{i}, k = 1, \ldots, j\right)
\end{equation}
where the $Z_i$'s are i.i.d. Exp($1$) rvs. The first three representations above are from Nagaraja (1982) who also shows that the joint limiting distribution $(W_1, \ldots, W_j)$ is identical to the joint distribution of the first $j$ lower record values from the cdf $G$.  The representation in (\ref{joint_ext3b}) is due to Hall (1978), and is more convenient when $G = G_3$.  
Thus, whenever $F \in D(G)$, the limiting form of the joint  distribution of the normalized spacings and the concerned extreme order statistics can be described as follows: \begin{eqnarray}\label{wrep}&&\hspace{-.3in}\left((X_{n:n}-X_{n-1:n})/b_n, \ldots, (X_{k+1:n}-X_{k:n})/b_n,  (X_{k:n} - a_n)/b_n, \right.\nonumber \\&&\hspace{-.3in}\left.(X_{k:n}-X_{k-1:n})/b_n,\ldots, (X_{k-s+1:n}-X_{k-s:n})/b_n)\right) \nonumber \\ &\stackrel{d}{\to}& (W_1 - W_2, \ldots, W_{n-k} -W_{n-k+1}, W_{n-k+1},\nonumber \\&& W_{n-k+1} -W_{n-k+2}, \ldots, W_{n-k+s} -W_{n-k+s+1})
\end{eqnarray} where the $W_j$'s have one of the forms given in (\ref{joint_ext1}) - (\ref{joint_ext3b}). We now specialize our results for each of the three domains.

\subsection{The Fr\'echet Domain} In this case $a_n$ can be 
chosen to be 0 and $b_n$ to be $F^{-1}(1-1/n)$ ($= x_{1-n^{-1}}$).  
 The representation in (\ref{wrep}) for the limiting joint distribution along with (\ref{joint_ext1})  suggests that  an extreme spacing is not asymptotically exponential, and the adjacent spacings are neither independent, nor identically distributed in the limit. The asymptotic independence of the spacings and the extreme order statistic also fails. Hence when $F\in D(G_{1; \alpha})$, the asymptotic distributional structure for the extreme spacings differs from that for the central and intermediate cases. 
 
 From (\ref{multext}) and (\ref{joint_ext1}) we conclude that when $F\in \mathcal{D}(G_{1;\alpha})$, the normalized higher order spacing, $$(X_{n:n} - X_{n-j:n})/b_n \stackrel{d}{\to} W_1 - W_{j+1}\stackrel{d}{=}Z_1^{-1/\alpha} - (Z_1+S_j)^{-1/\alpha},$$
 where the sum $S_j = Z_2 + \cdots + Z_{j+1}$ is a Gam($j$) rv that is independent of $Z_1$. This distributional representation complements the work of Pakes and Steutel (1997) who have given an expression for the cdf of the limiting rv as (p. 192)
 $$ P(W_1 \leq w) + P(W_1 -(S_j + W_1^{-\alpha})^{-1/\alpha} \leq w, W_1 > w), ~w > 0.$$
They comment that this expression for the cdf `does not seem susceptible to simplification for any choice of the parameter $\alpha$'; for the other two domains, they provide explicit distributional representation that is equivalent to ours (see below).
%%%%%%%%%%%%%%%%%%%
\subsection{The Weibull Domain}
Here $x_1 (=F^{-1}(1))$ is finite and can be chosen to be our $a_n$ and the scaling constant $b_n$ can be chosen to be $x_1 - x_{1-1/n}$. From (\ref{wrep}) and (\ref{joint_ext2}) we can conclude that the normalized adjacent spacings are asymptotically i.i.d.\ exponential iff $\alpha =1$ when $F \in \mathcal{D}(G_{2; \alpha})$.  Otherwise, they are all dependent.  When $\alpha = 1$ and $k < n$, the joint asymptotic distributional structure of $((X_{n:n}-X_{n-1:n})/b_n, \ldots, (X_{k+1:n}-X_{k:n})/b_n,  (X_{k:n} - a_n)/b_n, (X_{k:n}-X_{k-1:n})/b_n,\ldots,(X_{k-s+1:n}-X_{k-s:n})/b_n)$ is that of
\begin{equation}\label{freshrep}\left(Z_2, \ldots, Z_{n-k+1},-( Z_1 + \cdots + Z_{n-k+1}), Z_{n-k+2},\cdots, Z_{n-k+s+1}\right).\end{equation}
Thus when $\alpha = 1$, $X_{k:n}$ is  asymptotically independent of the spacings in its left neighborhood, but is symmetrically dependent on the ones on its right. This conclusion is formalized in the following result.
\begin{theorem}\label{ext_2}
When $F \in \mathcal{D}(G_{2; \alpha=1})$, for each fixed $n-k$ and $s$, the asymptotic joint distribution of $((X_{n:n}-X_{n-1:n})/b_n, \ldots, (X_{k+1:n}-X_{k:n})/b_n,  (X_{k:n} - a_n)/b_n, (X_{k:n}-X_{k-1:n})/b_n,\ldots, (X_{k-s+1:n}-X_{k-s:n})/b_n)$ has the representation given by (\ref{freshrep}), where the $Z_j$'s are i.i.d. Exp($1$) rvs.  The $b_n$ can be chosen as $x_1-F^{-1}(1-1/n)$.  When $F \in \mathcal{D}(G_{2; \alpha})$ the adjacent spacings are i.i.d. exponential iff $\alpha = 1$. \end{theorem}

The standard uniform distribution is in $\mathcal{D}(G_{2;\alpha})$ with $\alpha = 1$ and hence has asymptotically i.i.d. extreme spacings.  We had reached this conclusion earlier in Lemma \ref{lemma1} (recall $x_1 - F^{-1}(1-1/n) = 1/n$).  But we have a more general result now that describes the symmetric dependence of the right neighborhood spacings on $X_{k:n}$ and is applicable to all $F \in \mathcal{D}(G_{2;1})$.  In fact with $n-k$ fixed, given $(X_{k:n}- x_1)/b_n =u~  (< 0)$, $(X_{n:n}-X_{n-1:n})/b_n, \ldots, (X_{k+1:n}-X_{k:n})/b_n$ behave like the spacings from a random sample of size $n-k$ from a uniform distribution over $(u,0)$. For the form of the joint distribution, see, e.g., David and Nagaraja (2003; Sec. 6.4).

 From (\ref{multext}) and (\ref{joint_ext2})  we conclude that when $F\in \mathcal{D}(G_{2;\alpha})$, the normalized higher order spacing $$(X_{n:n} - X_{n-j:n})/b_n \stackrel{d}{\to} W_1 - W_{j+1}\stackrel{d}{=} (Z_1+S_j)^{1/\alpha} -Z_1^{1/\alpha},$$
 where the sum $S_j = Z_2 + \cdots + Z_{j+1}$ is a Gam($j$) rv that is independent of $Z_1$. This is the conclusion of Theorem 7.2 in Pakes and Steutel (1997).

 %%%%%
\subsection{The Gumbel Domain}  Using (\ref{wrep}) along with the representation for $W_j$ in  (\ref{joint_ext3b}), we conclude the following. 
\begin{theorem}\label{ext_3}
When $F \in \mathcal{D}(G_{3})$, for each fixed $n-k$ and $s$, the asymptotic joint distribution of $((X_{n:n}-X_{n-1:n})/b_n, \ldots, (X_{k+1:n}-X_{k:n})/b_n,  (X_{k:n} - a_n)/b_n, (X_{k:n}-X_{k-1:n})/b_n,\ldots, (X_{k-s+1:n}-X_{k-s:n})/b_n)$ has the following distributional representation:
\begin{equation}\label{ext_3b}
\left(Z_1, \frac{Z_2}{2}, \cdots, \frac{Z_{n-k}}{n-k},~ \sum_{i=n-k+1}^{\infty} \frac{Z_i -1}{i} + \gamma - \sum_{i=1}^{n-k} \frac{1}{i},~
\frac{Z_{n-k+1}}{n-k+1}, \cdots, \frac{Z_{n-k+s}}{n-k+s}\right)
\end{equation}
where the $Z_i$'s are i.i.d.\ Exp($1$) rvs.  The $b_n$ can be chosen as $m(F^{-1}(1-1/n))$, and if (\ref{sG3}) holds, as $\{n f(F^{-1}(1-1/n))\}^{-1}$.  
\end{theorem}

The representation in (\ref{ext_3b}) shows that while $X_{k:n}$ is independent of spacings in its right neighborhood, it is correlated with the spacings in its left neighborhood, and this correlation decreases at the rate of $1/(n-k+j), j \geq 1$ as one moves away from it.   This is in contrast with the situation when $F \in \mathcal{D}(G_{2,1})$.  

Weissman (1978) has considered the limit distribution of  $$((X_{n:n}-X_{n-1:n})/b_n, \ldots, (X_{k+1:n}-X_{k:n})/b_n,  (X_{k:n} - a_n)/b_n),$$  and has given the representation given by the first $(n-k)$ components of the vector in (\ref{ext_3b}). He has also noted the independence of these spacings and $W_{n-k+1}$ (in his Theorem 2).

 The above theorem shows that by using varying scaling sequences for the spacings, we can obtain i.i.d. standard exponential distributions in the limit. In particular we have the following:
$$((n-j+1)(X_{n-j+1:n}-X_{n-j:n})/b_n, j = 1, \ldots, n-k+s) \stackrel{d}{\to} (Z_1, \cdots, Z_{n-k+s}).$$ 

 From (\ref{multext}) and (\ref{joint_ext3b})  or from the representation (\ref{ext_3b}) we can conclude that when $F\in \mathcal{D}(G_{3})$, the normalized higher order spacing, $$(X_{n:n} - X_{n-j:n})/b_n \stackrel{d}{\to} W_1 - W_{j+1}\stackrel{d}{=} Z_1 + \frac{Z_2}{2} + \cdots + 
 \frac{Z_j}{j} \stackrel{d}{=} Z_{j:j}.$$The last equality above follows from the representation for exponential order statistics given in (\ref{exp_os}). This is the conclusion reached in Theorem 7.1 of Pakes and Steutel (1997).  
%%%%

\subsection{Remarks - The Extreme Case}
\subsubsection{Cox Processes}
Whenever $F \in \mathcal{D}(G)$, we can say that the arrival process of order statistics in the left neighborhood of an upper extreme order statistic is asymptotically a Poisson process and is independent of its value only when $F \in G_2$ with $\alpha = 1$.  The arrival processes on both sides of the extreme order statistics are pure birth processes when $F \in G_3$; only the arrival process on the right side is independent of the order statistic.  

Harshova and H\"usler (2000) have shown that the arrival processes on the left neighborhood of the sample maximum are special  {\it Cox processes}, when $G$ is of Weibull ($G_{2; \alpha}$) or Gumbel ($G_3$) type cdf. Cox processes are  mixed Poisson processes where the time-dependent intensity $\lambda(t)$ is itself a stochastic process (Daley and Vere-Jones (2003; Sec.\ 6.2)). Harshova and H\"usler consider the counting process in the left neighborhood of $X_{n:n}$,  $N_n(\cdot)$, defined by $N_n([a,b)) = \#\{X_i \in [X_{n:n} - a\cdot d_n, X_{n:n}- b\cdot d_n)\}, 0 < a < b$.  From their Theorem 1.2 it follows that $N_n(\cdot) \stackrel{d}{\to} N(\cdot)$ where $\{N(t), t >0\}$ is a Cox process with stochastic intensity function $\lambda(t) = e^{t-W}$ in the Gumbel case where $W$ has cdf $G_3$; and $\lambda(t) = \alpha(t-W)^{\alpha -1}$ in the Weibull case where $W$ has cdf $G_{2; \alpha}$. The cdfs $G_{2;\alpha}$ and $G_3$ are given in (\ref{extremeG}).  Representations given in  (\ref{joint_ext2}), and (\ref{joint_ext3a}) or (\ref{joint_ext3b}) provide another characterization of the resulting Cox processes in terms of the distribution of inter-arrival times of order statistics below the maximum. 

\subsubsection{Higher Order Extreme Spacings} The representation for the special higher order extreme spacing  involving the sample maximum ($X_{n:n}- X_{n-j:n}$, discussed above) can be expanded to other extremes.  From (\ref{multext}) - (\ref{joint_ext2}), and (\ref{joint_ext3b})  we conclude that  as $n \to \infty$, for fixed $1 \leq i < j$,  $(X_{n-i+1:n} - X_{n-j+1:n})/b_n $ converges in distribution to
\begin{eqnarray}\label{limitspacing}  (\sum_{l=1}^i Z_l)^{-1/\alpha} - (\sum_{l=1}^j Z_l)^{-1/\alpha} &\stackrel{d}{=}& W_i - (W_i^{-\alpha} + \mbox{Gam($(j-i)$))}^{-1/\alpha}, G \in \mathcal{D}(G_{1;\alpha}); \nonumber\\   (\sum_{l=1}^i Z_l)^{1/\alpha} - (\sum_{l=1}^j Z_l)^{1/\alpha}&\stackrel{d}{=}& (W_i^{\alpha} + \mbox{Gam($(j-i)$))}^{1/\alpha} -W_i, G \in \mathcal{D}(G_{2;\alpha}), \nonumber \\ &\stackrel{d}{=} & \mbox{Gam($(j-i)$)}, G \in \mathcal{D}(G_{2;1});\nonumber\\ \frac{Z_{i}}{i} +  \cdots +  \frac{Z_{j-1}}{j-1}&\stackrel{d}{=}& Z_{j-i:j-1}, G \in \mathcal{D}(G_3).\end{eqnarray}
Here the last distributional equality follows from (\ref{exp_os}).  The above representations are extremely helpful in providing the asymptotic distribution theory for the number of order statistics around a specified extreme order statistic.  This will be illustrated in the next section where all cases (central, extreme, and intermediate) will be considered.  %%%%%%%%%%%%%%%%%%%%%%%%%%%
%%%%%%%%%
\section{Counts of Observations Around an Order Statistic}
Consider the following count statistics that track the number of observations in the right and left neighborhoods of $X_{k:n}$:
 \begin{eqnarray*}\label{Kn}
 K_{-}(n,k, d) &=&\#\{j: X_{j} \in (X_{k:n} - d, X_{k:n}) \},\nonumber\\
 K_+(n,k, d) &=&\#\{j: X_{j} \in (X_{k:n}, X_{k:n}+d)\}.
 \end{eqnarray*}
 Clearly,
 \begin{eqnarray}\label{link2}
P(K_{-}(n,k, d) < i) &=& P(X_{k:n} - X_{k-i:n} > d), \nonumber \\P(K_{+}(n,k, d) < j) &=& P(X_{k+j:n} - X_{k:n} > d),
\end{eqnarray}
and thus the asymptotic distribution theory for spacings developed here can be directly applied to determine the limit distributions of the count statistics for appropriately chosen $d$ that is dependent on $n$.  Pakes and Steutel (1997) have used the link in (\ref{link2}) in the reverse direction in the extreme case where they derive the limit distribution of 
$K_{-}(n,n, d_n)$ first and use it to determine the limit distribution of the spacing $X_{n:n} - X_{n-k:n}$. 

As noted in the introduction, the literature on the investigation into the limit distribution of $K_{-}$ and $K_{+}$ is substantial. Poisson limits are generally obtained when $d = d_n$  is nonrandom but is dependent on the behavior of $F$ around the concerned statistic.  We now discuss implications of our results on spacings on the asymptotic distribution of counts and compare our results with only the most relevant results in the literature.

\subsection{The Central and Intermediate Cases - The Poisson Counts}
We have seen in Theorems \ref{theorem1} and \ref{theorem3} that  the $(X_{k+i:n} - X_{k+i-1:n})/c_n$ are asymptotically i.i.d.\ standard exponential for any fixed (positive or negative) integer $i$, where $c_n = 1/nf(x_{p_n})$ with $p_n \equiv p = \lim(k/n) \in (0,1)$ in the central case, and in the intermediate case,  $p_n = k/n \to 0$ or $1$ such that, respectively $k$ or $n-k \to \infty$. In other words, $K_{-}(n,k, \lambda_1 c_n)$ and $K_{+}(n,k, \lambda_2 c_n)$ are asymptotically independent, and Poi($\lambda_1$) and  Poi($\lambda_2$) rvs, respectively.  This conclusion matches with that of Pakes (2009) who has established the asymptotic Poisson property under a variety of technical conditions for the central [Theorem 16]  and intermediate cases [Theorems 9, 10 (b), 12 (b)]. 

Dembi\'nska et al. (2007) have established the asymptotic Poisson nature of both $K_{-}$ and $K_{+}$ for $d_n$ that depends on the local property of the quantile function $F^{-1}$ in the intermediate case and an associated technical condition on $F$ in the central case.  Their condition for Poisson convergence of $K_{+}(n,k, d_n)$ for a central order statistic [(17) in their Theorem 5.2] is 
\begin{equation}\label{Demb2}
{\lim}_{(x,y) \to (x_p,0^+)}\frac{F(x+y)-F(x)}{F(x_p+y) -F(x_p)} = 1. 
\end{equation}
This is comparable to our condition (\ref{demb}), but there are differences in these conditions and their implications.  While (\ref{demb}) specifies the behavior of the quantile function $F^{-1}$ around $p$, (\ref{Demb2})  puts a similar condition on the property of the cdf $F$ around $x_p$. Dembi\'nska et al.'s neighborhood is determined by $d_n = F^{-1}(p + \lambda/n) -x_p$, a quantity dependent on the behavior of $F^{-1}$ at $(p + \lambda/n)$.  In contrast, our $d_n = \lambda/n(f(x_p))$ depends  on the behavior of $F^{-1}$ only at $x_p$.  When $f$ is continuous around $x_p$ and $f(x_p)$ is positive and finite,  from L'Hospital's rule it follows that (\ref{Demb2}) is readily satisfied, and 
$ F^{-1}(p + \lambda/n) -x_p \sim \lambda/n(f(x_p)).$

Under (\ref{Demb2}) and a similar condition on the left neighborhood, Dembi\'nska and Balakrishnan (2010) have established the asymptotic independence of $K_{-}$ and $K_{+}$. This independence readily follows from our Theorems \ref{theorem2} and \ref{theorem3}.
The technical conditions used in Dembi\'nska et al. (2007) in the intermediate case (in their Theorem 6.1, for example) appear difficult to verify whereas the familiar von Mises conditions needed here are known to hold for many common distributions.  In addition, our results show that counts in disjoint intervals are Poisson and independent, and also that these are independent of the location of $X_{k:n}$.  These  finer conclusions on the limiting structure of the neighborhood cannot be reached using any of the currently available results in the literature on count statistics for the central and intermediate cases.
\subsection{The Upper Extremes - NonPoisson and Poisson Counts}
Asymptotic distributions of $K_{-}(n, k, d)$ and $K_{+}(n, k, d)$ have been investigated by many authors when $k$ or $n-k$ are held fixed starting from the work of Pakes and Steutel (1997) who looked at $K_{-}(n, n, d).$ Assuming $k$ is held fixed, Pakes and Li (1998) showed that $K_{-}(n, n- k, d)$ is asymptotically negative binomial, and Balakrishnan and Stepanov (2005) showed that  $K_{+}(n, n- k, d)$ is asymptotically binomial.  The success probability in these distributions is given by
$$\beta(d) = \lim_{x \to x_1}\frac{\overline{F}(x+d)}{\overline{F}({x})} \in (0,1),$$
where $x_1$ is assumed to be infinite.  

Pakes (2009) has considered the limit distribution of $K_{+}(n, n- k, cb_n)$ with $k$ fixed assuming that $F$ is in the domain of attraction of either Fr\'echet or Gumbel distribution and the $b_n$'s form the associated scaling sequence. When $F \in \mathcal{D}(G_{1; \alpha})$ he shows that the limit distribution of $K_{+}(n, n- k, b_n)$ is mixed binomial with parameters $k$ and random success probability that is a function of a Gam($k+1$) rv (his Theorem 5, part (a)). When $F \in \mathcal{D}(G_{3})$, $K_{+}(n, n- k, \lambda b_n)$ is shown to be asymptotically a Binomial rv with parameters $k$ and success probability $1- e^{-\lambda}$ (his Theorem 4).  

We now examine the consequences of the representations in (\ref{limitspacing}) and the relations in (\ref{link2}). When $k$ is fixed and $F \in \mathcal{D}(G_3)$, for any $j, 1 \leq j  \leq k$,
\begin{eqnarray*}
P(K_{+}(n, n-k, \lambda b_n) <j) &=& P(X_{n-k+j:n} - X_{n-k:n} >\lambda b_{n}) \\&\to& P( Z_{j:k} > \lambda) = \sum_{i=0}^{j-1}\binom{k}{i}(1- e^{-\lambda})^i (e^{-\lambda})^{k-i},\end{eqnarray*}
as $n\to \infty$. Since the maximum value attainable is $k$, the limit distribution is binomial, a result noted above. Further, 
\begin{eqnarray*}
P(K_{-}(n, n-k, \lambda b_n) <j) &=& P(X_{n-k:n} - X_{n-k-j:n} >\lambda b_{n}) \\&&\hspace{-.5in}\to P( Z_{j:k+j} > \lambda) = \sum_{i=0}^{j-1}\binom{k+j}{i}(1- e^{-\lambda})^i (e^{-\lambda})^{k+j-i},\end{eqnarray*}
resulting in a negative binomial distribution, a result shown by  Pakes and Li (1998).  They also derive the limit distribution of $K_{-}(n, n-k, \lambda b_n)$ in other cases; the representations in (\ref{limitspacing}) along with (\ref{link2}) yield us the same results.  Of these, a commonly known distribution is obtained only when $F \in \mathcal{D}(G_{2; 1})$ in which case $K_{+}(n, n-k, \lambda b_n)$ has a censored Poi($\lambda$) distribution that is censored on the right at $k$; this conclusion was reached in Theorem 4.1 of Dembi\'nska et al. (2007) under a set of technical conditions similar to the one given in (\ref{Demb2}). Further, $K_{-}(n, n-k, \lambda b_n)$ will have a Poi($\lambda$) distribution and these two statistics are asymptotically independent.

Whenever $F\in \mathcal{D}(G_{1; \alpha}~ \mbox{or}~ G_{2; \alpha \neq 1})$, we can obtain the asymptotic distributions of $K_{-}$ and $K_{+}$ using (\ref{limitspacing}) and (\ref{link2}) directly.  For example, when $F \in \mathcal{D}(G_{1; \alpha})$ we can use the corresponding representation in  (\ref{limitspacing}) to obtain the cdf of $K_{+}(n, n- k, b_n)$ in terms of Gamma rvs (in contrast with the mixed binomial representation of Pakes (2009) mentioned earlier).  While closed form expression for the cdf may not be available, the needed probabilities can be evaluated using tractable univariate integrals that involve gamma type integrands that can be easily evaluated numerically. The link between Gamma and Poisson cdfs comes in handy in this simplification. 
\section{Discussion}
We now provide further illustrations of applications of our results to distribution theory and inference.
\subsection{Examples}
Our examples thus far were the uniform and exponential populations, but our results are widely applicable since the conditions imposed here are satisfied by several common distributions.  In the central case, we need positivity and continuity of the population pdf at $x_p$ to achieve independent Poisson arrival process in both right and left neighborhoods. Von Mises conditions are satisfied by the common distributions that are in the domain of attraction of an extreme value cdf $G$ (given in (\ref{extremeG})) and thus the intermediate case also leads to independent Poisson arrival process for these distributions.  The extreme case does not require the von Mises conditions, and provides interesting examples of situations where we do not get Poisson processes.  For example, for $F \in \mathcal{D}(G_{1; \alpha})$, a property satisfied by Pareto and loggamma distributions, the arrival process is no longer Poisson.  Tables 3.4.2-3.4.4 of Embrechts et al.\ (1997) contain a good list of distributions in the domain of attraction of each of the three extreme value distributions along with the necessary norming (scaling) constants needed for the application of our results in the extreme case. 

Our intermediate and extreme case discussions focussed on the upper end of the sample.  Parallel results hold for the lower end of the sample and  upper-end and lower-end spacings can exhibit different types of clustering processes. For example, in the exponential parent case, upper extremes are in the Gumbel domain, and the lower extremes are in the Weibull domain with $\alpha = 1$ (i.e., Exp ($1$)).  Thus, for the lower extremes, we have a homogeneous Poisson arrival process in the right neighborhood, whereas for the upper extremes, we have  a pure birth process in the left neighborhood of the concerned order statistic.
\subsection{Inferential Implications}
Theorem 2 (a) can be used in the central case to provide      (asymptotically) distribution-free estimates of $x_p$ and $f(x_p)$ as noted by Siddiqui (1960) when he studied the joint distribution of $X_{k:n}, X_{k+r:n}-X_{k:n}, X_{k:n}-X_{k-s:n}$ [See Sec.\ 3.3 Remarks].  It follows from Theorem 2 that $n f(x_p)(X_{k+r:n}-X_{k-s:n})$ is asymptotically Gam$(r+s)$ and this fact can be used to provide estimates of $f(x_p)$ and confidence intervals for the population pdf at the $p$th quantile. The asymptotic independence of  $n f(x_p)(X_{k+r:n}-X_{k-s:n})$ and $ \sqrt{n} f(x_p)(X_{k:n} - x_p)$ and their known familiar distributions can be used to find the distribution of the pivotal quantity
$$\frac{(X_{k:n} - x_p)}{\sqrt{n}(X_{k+r:n}-X_{k-s:n})}.$$
From Theorem 2 it follows that this rv behaves asymptotically as the ratio of a standard normal and an independent gamma rv (or a scaled chi-square rv)  and this distribution is free of $f(x_p)$.  It easily leads to an asymptotically distribution-free confidence interval for $x_p$. 

A similar application of Theorem 3 would provide asymptotically distribution-free inference for the intermediate population quantile $F^{-1}(k/n)$ and pdf $f(F^{-1}(k/n))$ when one of the Von-Mises conditions is assumed to hold.

In the extreme case, we have seen that the limit distributions of the top $k$ order statistics are dependent on the domain of attraction. Weissman (1978) has discussed in detail inference on tail parameters (extreme quantiles and the {\it tail index} $1/\alpha$) based on these limit results. 
\subsection{Concluding Remarks}
%%%%%%%%%%
 It is interesting to note that norming/scaling constants for $X_{k:n}$ and the adjacent spacings are of the same order only for the extreme case ($b_n$); the limiting distributions are similar as well (functions of Exp($1$) rvs).  For the central and intermediate cases, the spacings and $X_{k:n}$ are scaled differently, and their limit distributions are different; the spacings are related to exponential whereas $X_{k:n}$ relates to the normal.   For the extreme and intermediate cases our sufficient conditions that ensure the nondegenerate limit distributions for $X_{k:n}$  and for the adjacent spacings are the same.  In the central case, asymptotic normality for $X_{k:n}$ requires $k = np + o(\sqrt{n})$ (actually slightly less restriction on $k$ would work),  the asymptotic independence property of spacings holds whenever $k = np + o({n}).$
 
We have focussed here on neighborhoods of a single selected order statistic; this work can easily be extended to multiple neighborhoods.  In the case of two or more central order statistics and their neighborhoods, we obtain a multivariate normal limit distribution for the selected order statistics and independent Poisson processes around them.  Such a set up is considered in Theorem 3.1 of Dembi\'nska and Balakrishnan (2010) where the independence of Poisson counts in right and left neighborhoods of multiple central order statistics is derived.    In other cases (for example, one upper extreme, and another lower, considered in Theorem 2.1 of  Dembi\'nska and Balakrishnan (2010)), the resulting  counting processes will turn out to be independent.
\section*{Acknowledgements}
The authors thank two anonymous reviewers for a close reading of the manuscript and many helpful suggestions.   
%%%%%%%%%%%%

\end{document}